\documentclass{amsart}

\usepackage{amsfonts}
\usepackage{amsthm}
\usepackage{amstext}
\usepackage{amsmath}
\usepackage{amscd}
\usepackage{amssymb}
\usepackage[mathscr]{eucal}
\usepackage{graphicx}
\usepackage{graphics}
\usepackage{epsf}
\usepackage{array}
\usepackage{url}

\newcommand{\FigureKStar}{Figure~\ref{Fig:kstar}}
\newcommand{\FigureFreeArc}{Figure~\ref{Fig:freeArc}}

 \makeatletter
 \def\@seccntformat#1{\csname the#1\endcsname.\quad}
 \makeatother

\theoremstyle{plain}
\newtheorem{theorem}{Theorem}
\newtheorem{proposition}[theorem]{Proposition}
\newtheorem{corollary}[theorem]{Corollary}
\newtheorem{lemma}[theorem]{Lemma}

\newtheorem*{theoremMain}{Main Theorem}
\newcommand{\theoremMainRef}{Main~Theorem}
\newtheorem*{conjecture}{Conjecture}
\newtheorem*{question}{Question}

\theoremstyle{definition}
\newtheorem{definition}[theorem]{Definition}

\theoremstyle{remark}
\newtheorem{remark}[theorem]{Remark}

\theoremstyle{remark}
\newtheorem{example}[theorem]{Example}

\numberwithin{equation}{section}

\newcommand{\mL}{L\kern-0.08cm\char39}

\newcommand{\Lip}{\operatorname{Lip}}

\newcommand{\Cut}{\operatorname{Cut}}   

\newcommand{\NNN}{\mathbb N}

\newcommand{\RRR}{\mathbb R}
\newcommand{\SSS}{\mathbb S}
\newcommand{\AAA}{\mathbb{A}}
\newcommand{\III}{I}

\newcommand{\AAa}{\mathcal{A}}
\newcommand{\BBb}{\mathcal{B}}
\newcommand{\CCc}{\mathcal{C}}
\newcommand{\DDd}{\mathcal{D}}

\newcommand{\HHh}{\mathcal{H}}

\newcommand{\dist}{\operatorname{dist}}
\newcommand{\closure}[1]{\overline{#1}}
\newcommand{\boundary}[1]{\partial {#1}}
\newcommand{\interior}[1]{\operatorname{int}(#1)}

\newcommand{\eps}{\varepsilon}

\newcommand{\abs}[1]{\lvert#1\rvert}
\newcommand{\card}[1]{\##1}
\newcommand{\invlim}{\varprojlim}
\newcommand{\dimUpper}{\overline{\dim}}
\newcommand{\dimLower}{\underline{\dim}}
\newcommand{\toPar}[1]{\stackrel{#1}{\rightarrow}}
\newcommand{\length}[1]{\HHh^1(#1)}
\newcommand{\lengthd}[2]{\HHh^1_{#1}(#2)}

\newcommand{\IT}{I^{\operatorname{T}}}
\newcommand{\ID}{I^{\operatorname{D}}}
\newcommand{\IED}{I^{\operatorname{ED}}}

\begin{document}

\title[Entropy and exact Devaney chaos on totally regular continua]
 {Entropy and exact Devaney chaos \\on totally regular continua}

\author[Vladim\'\i r \v Spitalsk\'y]{Vladim\'\i r \v Spitalsk\'y}
\address{Department of Mathematics, Faculty of Natural Sciences,
          Matej Bel University, Tajovsk\'eho 40, 974 01 Bansk\'a Bystrica,
          Slovakia}
\email{vladimir.spitalsky@umb.sk}
\thanks{The author wishes to express his thanks
to \mL{}u\-bo\-m\'ir Snoha for his help with the preparation of the paper.
The author was supported by the Slovak Research and Development Agency
under the contract No.~APVV-0134-10 and by the Slovak Grant Agency under the grant
number VEGA~1/0978/11.}



\begin{abstract}
We study topological entropy of exactly Devaney chaotic maps on
totally regular continua, i.e.~on (topologically) rectifiable
curves.
After introducing the so-called $P$-Lipschitz
maps (where $P$ is a finite invariant set) 
we give an upper bound for their topological entropy.
We prove that if a non-degenerate totally regular
continuum $X$ contains a free arc which
does not disconnect $X$ or if $X$ contains arbitrarily large
generalized stars then $X$ admits an exactly Devaney chaotic map
with arbitrarily small entropy. A possible application
for further study of the best lower bounds
of topological entropies of transitive/Devaney chaotic maps
is indicated.
\end{abstract}

\maketitle

\thispagestyle{empty}

\section{Introduction}\label{S:intro}

A \emph{(discrete) dynamical system} is a pair $(X,f)$ where $X$ is
a compact metric space and $f:X\to X$ is a continuous map. For
$n\in\NNN$ we denote the composition $f\circ f\circ \dots\circ f$
($n$-times) by $f^n$. A point $x\in X$ is a \emph{periodic point} of
$f$ if $f^n(x)=x$ for some $n\in\NNN$. We say that a dynamical
system $(X,f)$ is \emph{(topologically) transitive} if for every
nonempty open sets $U,V\subseteq X$ there is $n\in\NNN$ such that
$f^n(U)\cap V\ne\emptyset$. If every iterate $f^n$ ($n\ge 1$) is
transitive we say that $f$ is \emph{totally transitive}. If
$(X\times X, f\times f)$ is transitive then $(X,f)$ is called
\emph{(topologically) weakly mixing}. A system $(X,f)$ is
\emph{(topologically) mixing} provided for every nonempty open sets
$U,V\subseteq X$ there is $n_0\in\NNN$ such that $f^n(U)\cap
V\ne\emptyset$ for every $n\ge n_0$. Further, $(X,f)$ is
\emph{(topologically) exact} or \emph{locally eventually onto} if
for every nonempty open subset $U$ of $X$ there is $n\in\NNN$ such
that $f^n(U)=X$. Finally, a system $(X,f)$ is \emph{Devaney chaotic}
(\emph{totally Devaney chaotic}, \emph{exactly Devaney chaotic})
provided $X$ is infinite, $f$ is transitive (totally transitive,
exact, respectively) and has dense set of periodic points.

One of the main quantitative characteristics of a dynamical system
$(X,f)$ is the \emph{topological entropy}, denoted by $h(f)$. A
challenging topic in discrete dynamics is the study of relationships
between entropy and qualitative properties of dynamical systems,
such as transitivity, exactness and (exact) Devaney chaos. This goes
back to 1982 when Blokh in \cite{Blo82} showed that every transitive
map on the unit interval $[0,1]$ has entropy at least $(1/2) \log
2$ and that this is the best possible bound, i.e.~$\IT([0,1])=(1/2) \log 2$, where
$$
 \IT(X)=\inf \{h(f):\quad f:X\to X \text{ is transitive}\}.
$$
Moreover, this infimum is in fact the minimum, i.e.~the infimum is
attainable.

Instead of $\IT(X)$ one can study infima/minima of entropies of more
restrictive classes of dynamical systems on $X$. Denote by $\ID(X)$
and $\IED(X)$ the infima of entropies of Devaney chaotic and exactly
Devaney chaotic systems on $X$, respectively. We see at once that
$\IT(X)\le \ID(X)\le \IED(X)$. For (connected) \emph{graphs}, these infima and
the existence of corresponding minima were studied e.g.~in
\cite{Blo87, AKLS, ABLM, ARR, ALM, Ye, Bal01, Ruette, KM, HKO11};
see Table~\ref{tab:ResultsOnGraphs} for some of the results. (For the definition
of classes $\mathcal{P}(i)$ of trees see \cite{Ye}.)
Every transitive graph map which is not conjugate to an
irrational rotation of the circle has dense periodic points
\cite{Blo84}, so $\ID(X)=\IT(X)$ whenever $X$ is a graph. Notice also
that any exact map on a non-degenerate space $X$ has positive
entropy, see e.g.~\cite{KM}; hence the infimum is never attained if
$\IED(X)=0$.

\begin{table}[ht!]
    \begin{tabular}
       {!{\vrule width 1pt}l!{\vrule width 1pt}c|c!{\vrule width 1pt}c|c!{\vrule width 1pt}}
    \noalign{\hrule height 1pt}
     The space $X$
     & $\IT(X)$ / $\ID(X)$ & Attainable?
     & $\IED(X)$ & Att.?
    \\
    \noalign{\hrule height 1pt}
     interval
     & $\frac 12\log 2$ & yes
     & $\frac 12\log 2$ & no
    \\
    \hline
     circle
     & $0$ & yes/no
     & $0$ & no
    \\
    \hline
     $n$-star
     & $\frac 1n \log 2$ & yes
     & $\frac 1n \log 2$ & ?
    \\
    \hline
     tree with $n$ ends
     & $\ge \frac 1n \log 2$ & depends on $X$
     & $\ge \frac 1n \log 2$ & ?
    \\
    \hline
     tree $\in\mathcal{P}(i)$, $n$ ends
     & $\le \frac 1{n-i} \log 2$ & depends on $X$
     & $\le \frac 1{n-i} \log 2$ & ?
    \\
    \hline
     graph $\ne$ tree, circle
     & $0$ & no
     & $0$ & no
    \\
    \noalign{\hrule height 1pt}
    \end{tabular}
\medskip
   \caption{Infima of entropies on graphs}
   \label{tab:ResultsOnGraphs}
\end{table}

In the present paper we study $\IED(X)$ for {totally regular
continua} $X$. Recall that a \emph{continuum} is a connected compact
metric space. A continuum
$X$ is \emph{totally regular} \cite{Nik}
if for every point $x\in X$ and every countable set $P\subseteq X$
there is a basis of neighborhoods of $x$ with finite boundary not
intersecting $P$. Equivalently
one can define totally regular continua as those continua $X$ which
admit a compatible metric $d$ such that $(X,d)$ is a
\emph{rectifiable curve} (i.e.~$(X,d)$ has finite one-dimensional
Hausdorff measure $\lengthd{d}{X}$). Every
\emph{graph} and even every \emph{local dendrite} (a locally
connected continuum containing only finitely many simple closed
curves) is totally regular. As another example one can mention
the Hawaiian earrings (an infinite wedge of circles), which is a totally
regular continuum but not a local dendrite.

Recall that an \emph{arc} is a homeomorphic image of $[0,1]$. A
\emph{free arc} in $X$ is an arc $A\subseteq X$ such that every
non-end point of $A$ is an interior point of $A$ (in the topology of
$X$). We say that a free arc $A\subseteq X$ \emph{does not
disconnect} $X$ if no non-end point of $A$ disconnects $X$; for locally
connected continua (hence also for totally regular ones) this is
equivalent to the fact that $A$ is a subset of a simple closed curve
in $X$, see e.g.~\cite[IV.9.3]{Why42}. We say that $X$
\emph{contains arbitrarily large generalized stars} if for any
$k\in\NNN$ there are a point $a\in X$ and $k$ components of
$X\setminus\{a\}$ such that the closures of the components are
homeomorphic relative to $a$ (i.e.~the corresponding homeomorphisms
fix the point $a$).

In \cite{SpA} we introduced the so-called length-expanding Lipschitz maps
(briefly LEL-maps) and we showed
that for every two non-degenerate totally regular continua $X,Y$
there exists a LEL-map $f:X\to Y$;
see Section~\ref{SS:tentLike} for details.
LEL-maps form ``building blocks'' in our constructions of exactly Devaney
chaotic systems with small entropy. Our main results are summarized in the following theorem.

\newcommand{\theoremMainText}{
 For every non-degenerate totally regular continuum $X$ it holds that $\IED(X) < \infty$.
 Moreover, if $X$ contains arbitrarily large generalized stars or
 $X$ contains a free arc which does not disconnect $X$, then
 $X$ admits
\begin{itemize}
 \item a Devaney chaotic map which is not totally Devaney chaotic, and
 \item an exactly Devaney chaotic map,
\end{itemize}
both with arbitrarily small positive entropy;
hence $\IED(X)=0$.
}

\begin{theoremMain}
\theoremMainText{}
\end{theoremMain}

Thus, in particular, $\IED(X)$ and hence also $\ID(X)$ and $\IT(X)$
are zero for many dendrites,
e.g.~for the $\omega$-star or the Wazewski's universal dendrite.
Still, the following general question is open.

\begin{question}[Baldwin \cite{Bal01}]
Is it true that $\IT(X)=0$ for every dendrite $X$ which is not a
tree?
\end{question}

Our result gives $\IED(X)=0$ also for many spaces $X$ which are not
dendrites, say for the Hawaiian earrings.
Since every graph which is not a tree contains a free
non-disconnecting arc, Main Theorem is a generalization of
\cite[Theorems~3.7 and 4.1]{ARR}, cf.~the last row of
Table~\ref{tab:ResultsOnGraphs}. Let us note that $\IED(X)=0$ is
true also for other subclasses of the class of totally regular
continua, see e.g.~Lemma~\ref{T:I(free-arc)} and
Remark~\ref{R:generalizations}. We even believe that the following
conjecture is true.

\begin{conjecture} $\IED(X)=0$ for every totally regular continuum $X$ which is
not a tree.
\end{conjecture}

A potential application of our results lies in the fact that the
estimates of $\IT(X)$, $\ID(X)$ and $\IED(X)$ for one-dimensional
continua can be used to estimate the corresponding quantities for
some higher-dimensional continua. To do that, one needs two tools. The
first one is an appropriate extension theorem. For instance, in
\cite{AKLS} and \cite{KM11} it was proved that a transitive map
$f:X\to X$ on a compact metric space $X$ without isolated points can
be extended to a transitive map $F$ on $X\times [0,1]$ without
increasing the entropy. An analogous result is true for Devaney
chaotic maps \cite{BS}. The second tool are procedures enabling to
go from $X\times[0,1]$ to the cone or the suspension over $X$, or
to the space $X\times\SSS^1$, similarly as in \cite{BS}. However,
it is not the topic of the present paper to continue in this direction.

Finally let us notice that the study of infima of entropies of
transitive/Devaney chaotic maps on a given space is naturally
connected with the question whether the space admits a \emph{zero entropy}
transitive/Devaney chaotic map. Using techniques developed in
\cite{BS} one can easily construct a Devaney chaotic map on the
Cantor fan (the cone over the Cantor set) which has zero entropy. On
the other hand, if $X$ is a compact metric space containing a free
arc and $X$ is not a union of finitely many disjoint circles, then
every transitive map on $X$ has positive entropy \cite{DSS}.
However, the following question is still open.

\begin{question}[Baldwin \cite{Bal01}]
Is there a dendrite admitting a transitive map with zero entropy?
\end{question}

The paper is organized as follows. In Section~\ref{S:preliminaries}
we recall all the needed definitions and facts. In
Section~\ref{S:entropyBound} we introduce the so-called $P$-Lipschitz maps
and we give an upper bound for their topological entropy.
Finally, in
Section~\ref{S:applications} we prove \theoremMainRef{}.

\section{Preliminaries}\label{S:preliminaries}
Here we briefly recall all the notions and results which will be needed
in the rest of the paper. The terminology is taken mainly from \cite{Kur2, Nad, Macias, Fal}.

If $M$ is a set, its cardinality is denoted by $\card M$.
The cardinality of infinite countable sets is denoted
by $\aleph_0$. If $M$ is a singleton set we often identify it with its only point.
We write $\NNN$ for the set of positive integers $\{1,2,3,\dots\}$,
$\RRR$ for the set of reals and $\III$ for
the unit interval $[0,1]$.
By an interval we mean any nonempty connected subset of $\RRR$
(possibly degenerate to a point).
For intervals
$J,J'$ we write $J\le J'$ if $t\le s$ for every $t\in J$, $s\in J'$.

By a \emph{space} we mean any nonempty metric space. A space is
called \emph{degenerate} provided it has only one point; otherwise
it is called \emph{non-degenerate}. If $E$ is a subset of a space
$X=(X,d)$ we denote the closure, the interior and the boundary of
$E$ by $\closure{E}$, $\interior{E}$ and $\boundary{E}$,
respectively, and we write $d(E)$ for the diameter of $E$. We say
that two sets $E,F\subseteq X$ are \emph{non-overlapping} if they
have disjoint interiors. For $x\in X$ and $r>0$ we denote the closed
ball with the center $x$ and radius $r$ by $B(x,r)$.

Let $(X,f)$ be a dynamical system. A set $A\subseteq X$ is called
\emph{$f$-invariant} if $f(A)\subseteq A$. For the definitions of (total)
transitivity, (weak, strong) mixing, exact\-ness and (exact) Devaney
chaos see Section~\ref{S:intro}.

\subsection{Minkowski dimension and topological entropy}\label{SS:minkDim}

A subset $S$ of
a metric space $(X,d)$ is called \emph{$\eps$-separated} if
$d(x,y)>\eps$ for every points $x\ne y$ from $S$.
A subset $R$ of $X$ is said
to \emph{$\eps$-span} $X$ if for every $x\in X$ there is $y\in R$
with $d(x,y)\le \eps$. We will denote by $s(X,\eps)$ and $r(X,\eps)$
the maximal cardinality of $\eps$-separated set and the minimal cardinality
of $\eps$-spanning set, respectively. We see at once that
$r(X,\eps) \le s(X,\eps)\le r(X,\eps/2)$.
The \emph{upper} and \emph{lower Minkowski (ball) dimensions} of $X$ are defined by
(see e.g \cite[5.3]{Mattila} or \cite[Definition~3.2.7]{Katok})
$$
 \dimUpper(X)
  = \limsup\limits_{\eps\to 0} \frac{\log r(X,\eps)}{-\log \eps}
\qquad\text{and }\qquad
 \dimLower(X)
  = \liminf\limits_{\eps\to 0} \frac{\log r(X,\eps)}{-\log \eps}
\,.
$$

Let $(X,f)$ be a dynamical system.
For $n\in\NNN$ define a metric $d_n$ on $X$ by
$d_n(x,y) = \max\limits_{0\le i<n} d(f^i(x),f^i(y))$.
We say that a set $S\subseteq X$ is \emph{$(n,\eps)$-separated} if
it is an $\eps$-separated subset of $(X,d_n)$, i.e.~if for every
distinct $x,y\in S$ there is $i<n$ with $d(f^i(x),f^i(y))>\eps$.
A set $R\subseteq X$ is said to \emph{$(n,\eps)$-span} $X$
if it $\eps$-spans $(X,d_n)$.
Let
$s_n(X,\eps,f)$ and $r_n(X,\eps,f)$ denote the maximal cardinality of $(n,\eps)$-separated
subsets and the minimal cardinality of $(n,\eps)$-spanning subsets of $X$, respectively.
Then, by \cite{Bowen71} and \cite{Dinaburg},
the \emph{topological entropy $h(f)$} of the
system $(X,f)$ is given by
\begin{equation}\label{EQ:BowenDefEntropy}
 h(f)
 = \lim\limits_{\eps\to 0} \limsup\limits_{n\to\infty} \frac 1n \log r_n(X,\eps,f)
 = \lim\limits_{\eps\to 0} \limsup\limits_{n\to\infty} \frac 1n \log s_n(X,\eps,f) \,.
\end{equation}

\subsection{Continua}\label{SS:continua}
A \emph{continuum} is a connected compact metric space.
A \emph{cut point} (or a \emph{separating point}) of a continuum $X$ is any point $x\in X$ such that
$X\setminus\{x\}$ is disconnected. A point $x$ of a continuum $X$
is called a \emph{local separating point} of $X$ if there is a connected
neighborhood $U$ of $x$ such that $U\setminus\{x\}$ is not connected.
If $a, b$ are points of $X$ then
any cut point of $X$ such that $a,b$ belong to different components of
$X\setminus\{x\}$ is said to \emph{separate $a,b$}. The set of all
such points is denoted by $\Cut(a,b)$ or $\Cut_X(a,b)$. If $a=b$ then obviously $\Cut(a,b)=\emptyset$.

Let $X$ be a continuum. A metric $d$ on $X$ is said to be \emph{convex}
provided for every distinct $x,y\in X$
there is $z\in X$ such that $d(x,z)=d(z,y)=d(x,y)/2$.
By \cite[Theorem~8]{BingPartSet} every locally connected
continuum admits a compatible convex metric.

If $X$ is a continuum, we introduce the following notion.
A \emph{splitting} of $X$ is any system
$\AAa=\{X_1,\dots,X_n\}$ of non-degenerate subcontinua
covering $X$
such that $P_\AAa=\bigcup_{i\ne j} X_i\cap X_j$ is finite.
In a splitting $\AAa$, every $X_i\in\AAa$
is \emph{regular closed} (i.e.~$X_i$ is closed and $\interior{X_i}$ is dense in it),
$\interior{X_i}=X_i\setminus\left( \bigcup_{j\ne i} X_j \right)\supseteq
X_i\setminus P_\AAa$ and $\partial{X_i}\subseteq P_\AAa$.

\begin{lemma}\label{L:convexMetricOnUnion}
Let $X$ be a continuum and $\AAa=\{X_1,\dots,X_n\}$ be a splitting of $X$.
For $i=1,\dots,n$
let $d_i$ be a convex metric on $X_i$ satisfying
$d_i(x,y)=d_j(x,y)$ whenever $i\ne j$ and $x,y\in X_i\cap X_j$.
For $x,y\in X$ put
\begin{equation}\label{EQ:convexMetricOnUnion}
\begin{split}
 d(x,y) = \inf\Bigg\{
   &\sum_{k=1}^m d_{i_k}(x_{k-1},x_k):\
   m\ge 1,\
   x_0=x,\ x_m=y
 \\
 &\text{ and }
   x_{k-1},x_k\in X_{i_k} \text{ for  } k=1,\dots,m
 \Bigg\}.
\end{split}
\end{equation}
Then $d$ is a convex metric on $X$ equivalent with the original one
and the infimum in (\ref{EQ:convexMetricOnUnion}) is in fact the minimum.
Moreover, if $d(x,y)=d_i(x,y)$ whenever $x,y\in P_\AAa\cap X_i$, then
$$
 d|_{X_i\times X_i} = d_i.
$$
\end{lemma}
\begin{proof}
Using convexity of every $d_i$ we may assume that, in the infimum from (\ref{EQ:convexMetricOnUnion}), $i_{k}\ne i_{k+1}$, $x_k\in P_\AAa$
for $1\le k<m$ and $x_k\ne x_l$ for every $k\ne l$.
Since $P_\AAa$ is finite we have that, for any fixed $x,y$, the infimum is in fact the minimum
and there are $m\ge 1$, $i_1,\dots,i_m$, $x_0=x,x_1,\dots,x_m=y$ such that
$x_0,\dots,x_{m}$ are pairwise distinct,
$x_0\in X_{i_1}$, $x_m\in X_{i_m}$,
\begin{equation}\label{EQ:convexMetricOnUnion2}
 d(x,y) =
   \sum_{k=1}^m d_{i_k}(x_{k-1},x_k)
 \quad
 \text{and}
 \quad
 x_k\in X_{i_k}\cap X_{i_{k+1}}\cap P_\AAa
 \quad \text{for } 1\le k< m.
\end{equation}

The fact that $d$ is a convex pseudometric is straightforward
and, by (\ref{EQ:convexMetricOnUnion2}), $d(x,y)=0$ implies $x=y$.
To show that
$d$ is equivalent with the original metric, fix a sequence $(x_k)_k$ and a point $x$;
since $X$ is compact we only need to prove that $x_k\to x$ implies $d(x_k,x)\to 0$. It is no loss of
generality in assuming that, for some $i$, $x_k\in X_i$ for every $k$.
Since $x_k\to x$, also $x\in X_i$ and $d(x_k,x)\le d_i(x_k,x)\to 0$.

Now fix $i$ and assume that $d(x,y)=d_i(x,y)$ for every $x,y\in P_\AAa\cap X_i$.
Take arbitrary $x,y\in X_i$. Then $d(x,y)\le d_i(x,y)$ by (\ref{EQ:convexMetricOnUnion}).
To show the opposite inequality take
$i_1,\dots,i_m$ and $x_0=x,x_1,\dots,x_m=y$ satisfying (\ref{EQ:convexMetricOnUnion2});
we can assume that $i_1=i_m=i$ (so possibly $x_0=x_1$ or $x_{m-1}=x_m$).
Since $x_1,x_{m-1}\in P_\AAa\cap X_i$, we get $d(x_1,x_{m-1})=d_i(x_1,x_{m-1})$
and
$$
 d_i(x,y)
 \le d_i(x,x_1)+d_i(x_1,x_{m-1}) + d_i(x_{m-1},y)
 \le \sum_{k=1}^m d_{i_k}(x_{k-1},x_k)=d(x,y) \,.
$$
Thus $d(x,y)=d_i(x,y)$ for every $x,y\in X_i$.
\end{proof}

\subsection{Hausdorff one-dimensional measure}\label{SS:hausdorffMeasure}
For a Borel subset $B$ of a metric
space $(X,d)$ the \emph{one-dimensional Hausdorff measure} of $B$ is defined by
$$
 \lengthd{d}{B}=\lim_{\delta\to 0} \lengthd{d,\delta}{B},
 \quad
 \lengthd{d,\delta}{B} = \inf\left\{
   \sum_{i=1}^\infty d(E_i):\ B\subseteq \bigcup_{i=1}^\infty E_i,\ d(E_i)<\delta
 \right\}.
$$
We say that $(X,d)$ has \emph{finite length} if $\lengthd{d}{X}<\infty$.
By e.g.~\cite[Proposition~4A]{Fre92},
\begin{equation}\label{EQ:length-diam}
\lengthd{d}{C}\ge d(C)
\qquad
\text{whenever }
C
\text{ is a connected Borel subset of }
X.
\end{equation}
If $A\subseteq X$ is an arc then $\lengthd{d}{A}$ is equal to the length of $A$
\cite[Lemma~3.2]{Fal}.
In the case when $(X,d)$ is the Euclidean real line $\RRR$
and $J\subset \RRR$ is an interval, $\lengthd{d}{J}$
is equal to the length of $J$ and we denote it simply by $\abs{J}$.

If $(X,d)$ is a continuum of finite length endowed with a convex metric $d$,
then it has the
so-called \emph{geodesic property} (see e.g.~\cite[Corollary~4E]{Fre92}):
for every distinct $x,y\in X$
there is an arc $A$ with end points $x,y$ such that $d(x,y)=\lengthd{d}{A}$;
any such arc $A$ is called a \emph{geodesic arc} or shortly a \emph{geodesic}.
Every subarc of a geodesic is again a geodesic. If $x,y$ are the end points
of a geodesic $A$ and $z\in A$ then $d(x,y)=d(x,z)+d(z,y)$.

\subsection{Lipschitz maps}\label{SS:lipschitzMaps}

A map $f:(X,d)\to (Y,\varrho)$ between metric spaces is called
\emph{Lipschitz} with a Lipschitz constant $L\ge 0$, shortly
\emph{Lipschitz-$L$}, provided $\varrho(f(x),f(x')) \le L\cdot
d(x,x')$ for every $x,x'\in X$; the smallest such $L$ is denoted by
$\Lip(f)$ and is called the \emph{Lipschitz constant} of $f$. If
$f:X\to Y$ is Lipschitz-$L$ then
$\lengthd{\varrho}{f(B)} \le L\cdot \lengthd{d}{B}$
for every Borel set $B\subset X$ such that $f(B)$ is Borel-measurable
\cite[p.~10]{Fal}. We will use the following simple fact.

\begin{lemma}\label{L:pwLip}
Let $X$ be a continuum with a convex metric $d$ and let
$\AAa=\{X_1,\dots,X_n\}$ be a splitting of $X$.
Let $(Y,\varrho)$ be a
metric space. Then any map $f:(X,d)\to(Y,\varrho)$ with
Lipschitz-$L$ restrictions $f|_{X_i}$ ($i=1,\dots,n$) is
Lipschitz-$L$.
\end{lemma}
\begin{proof}
By Lemma~\ref{L:convexMetricOnUnion},
for any $x,y\in X$
there are $x_0=x,x_1,\dots,x_{m-1},x_m=y$ such that
$x_{k-1},x_k\in X_{i_k}$ ($k=1,\dots,m$) and
$d(x,y)=\sum_{k=1}^m d(x_{k-1},x_k)$.
So
$$
 \varrho(f(x),f(y))
 \le \sum_{k=1}^m \varrho(f(x_{k-1}),f(x_k))
 \le L\cdot \sum_{k=1}^m d(x_{k-1},x_k)
 = L\cdot d(x,y).
$$
\end{proof}

\subsection{Totally regular continua}\label{SS:totallyRegular}
By e.g.~\cite{Kur2,Nik}, a
continuum $X$ is called
\begin{itemize}
    \item a \emph{dendrite} if it is locally connected and contains no simple closed curve;
    \item a \emph{local dendrite} if it is locally connected and contains at most finitely many
     simple closed curves;
    \item \emph{completely regular} if it contains no
     non-degenerate nowhere dense subcontinuum;
    \item \emph{totally regular} if for every $x\in X$ and every countable set $P\subseteq X$
     there is a basis of neighborhoods of $x$ with finite boundary not intersecting $P$;
    \item \emph{regular} if every $x\in X$ has a basis of neighborhoods with
     finite boundary;
    \item \emph{hereditarily locally connected} if every subcontinuum of $X$ is locally
     connected;
    \item a \emph{curve} if it is one-dimensional.
\end{itemize}
Notice that (local) dendrites as well as completely regular continua
are totally regular and (totally) regular continua are hereditarily locally connected, hence they are locally connected curves.

We will need the following simple estimate of the maximal cardinality of $\eps$-separated
sets in totally regular continua $X$. Notice that an analogous estimate follows from
the fact that the $1$-dimensional Minkowski content is equal to $\length{X}$ for
continua $X\subseteq \RRR^n$ of finite length (\cite[Theorem~3.2.39]{Federer}, see also \cite[p.~75]{Mattila}).

\begin{lemma}\label{L:separatedSetsInContinuaOfFiniteLength}
 Let $X=(X,d)$ be a non-degenerate totally regular continuum. Then
 $$
  \dimLower(X)=\dimUpper(X)=1
  \quad\text{and}\quad
  s(X,\eps) \le \frac{2\cdot\length{X}}{\eps}
  \quad\text{for every }0<\eps<d(X)\,.
 $$
\end{lemma}

\begin{proof}
Let us first prove the inequality for $s(X,\eps)$.
There is no loss of generality in assuming $\length{X}<\infty$.
By (\ref{EQ:length-diam}), 
\begin{equation}\label{EQ:lengthOfXcapBall}
 \length{B(x,r)} \ge r
 \qquad
 \text{for every } x\in X
 \text{ and } 0<r<d(X)/2 \,.
\end{equation}
Take $0<\eps<d(X)$ and an
$\eps$-separated set $S=\{x_1,\dots,x_k\}$ of maximal cardinality $k=s(X,\eps)$.
Since $d(x_i,x_j)>\eps$
for every $i\ne j$, the closed balls $B(x_i,\eps/2)$ are disjoint.
Applying (\ref{EQ:lengthOfXcapBall}) we conclude that
$
 \length{X}
 \ge \sum_i \length{B(x_i,\eps/2)}
 \ge k \eps/2
$, so
$s(X,\eps)\le 2\length{X}/\eps$.

The equality $\dimLower(X)=\dimUpper(X)=1$ is a special case
of \cite[Theorem~3.2.39]{Federer}.
(Notice that $\dimUpper(X)\le 1$ immediately follows from the upper estimate
for $s(X,\eps)$.)
\end{proof}

The following lemma will be used in the proof of Lemma~\ref{P:exactLipschitz}.
\begin{lemma}\label{L:totRegCircle}
If $X$ is a totally regular continuum which is not a dendrite, then there are
points
$a,b\in X$ and subcontinua $X_0,X_1$ of $X$ such that
$X_0\cup X_1=X$, $X_0\cap X_1=\{a,b\}$ and $\Cut_{X_i}(a,b)$ is uncountable for $i=0,1$.
\end{lemma}
\begin{proof}
 By \cite{BNT} we can write $X$ as the inverse limit
 $\invlim (X_n,f_n)$, where $X_n$ are graphs and $f_n:X_{n+1}\to X_n$ are monotone;
 let $\pi_n:X\to X_n$ ($n\in\NNN$) be the natural projections.
 Since $X$ is not a dendrite,
 by \cite[Theorem~10.36]{Nad}
 there is $m$ such that $X_m$ contains a circle $S$. Let $C$ be the set of all points
 $x_m\in S$ which are not vertices of $X_m$ and are such that $\pi_m^{-1}(x_m)$
 are singletons. Then $S\setminus C$ is countable
 since every system of disjoint non-degenerate subcontinua of $X$ is countable.
 Take $a_m\ne b_m$ from $C$ and put $a=\pi_m^{-1}(a_m)$, $b=\pi_m^{-1}(b_m)$.
 Write $X_m$ as the union $X_{m,0}\cup X_{m,1}$
 of two non-degenerate subgraphs with $X_{m,0}\cap X_{m,1}=\{a_m,b_m\}$.
 Then $X_i=\pi_m^{-1}(X_{m,i})$ ($i=0,1$), being inverse limits of continua
 \cite[Proposition~2.1.17]{Macias},
  are subcontinua of $X$
 covering $X$ and $X_0\cap X_1=\pi_m^{-1}(X_{m,0}\cap X_{m,1}) = \{a,b\}$.
\end{proof}

\subsection{Length-expanding Lipschitz maps on totally regular continua}\label{SS:tentLike}
Here we recall the main results of \cite{SpA}.
We say that a family $\CCc$ of
non-degenerate subcontinua of $X$
is \emph{dense} if every nonempty open set in $X$ contains a member of $\CCc$.
By $\CCc_I$ we denote the system of all non-degenerate closed subintervals of $I$;
the Euclidean metric on $I$ is denoted by $d_I$.

\begin{definition}[\cite{SpA}] Let $X=(X,d)$, $X'=(X',d')$
be non-degenerate (totally regular) continua of finite length and let
$\CCc,\CCc'$ be dense systems of
subcontinua of $X,X'$, respectively.
We say that a continuous map
$f:X\to X'$ is \emph{length-expanding} with respect to $\CCc,\CCc'$
if there exists $\varrho>1$ (called \emph{length-expansivity constant} of $f$)
such that, for every $C\in \CCc$, $f(C)\in\CCc'$ and
\begin{equation}\label{EQ:defLengthExpanding}
	   \text{if} \quad
	   f(C)\ne X'
	   \qquad\text{then}\quad
	   \lengthd{d'}{f(C)} \ge \varrho\cdot \lengthd{d}{C}.
\end{equation}
Moreover, if $f$ is surjective and Lipschitz-$L$ we say
that $f:(X,d,\CCc)\to (X',d',\CCc')$ is \emph{$(\varrho,L)$-length-expanding Lipschitz}.
Sometimes we briefly say that $f$ is
\emph{$(\varrho,L)$-LEL} or only \emph{LEL}.
\end{definition}

\newcommand{\propExact}{
Let $f:(X,d,\CCc)\to (X,d,\CCc)$ be a LEL map. Then $f$ is exact and has finite
positive entropy. Moreover, if $f$ is the composition $\varphi\circ\psi$ of some maps
$\psi:X\to I$ and $\varphi:I\to X$, then $f$ has the
specification property and so	it is exactly Devaney chaotic.
}

\begin{proposition}[\cite{SpA}]\label{P:tentLikeIsExact}
\propExact{}
\end{proposition}

For $k\ge 1$ let $f_k:I\to I$ be the piecewise linear continuous map which fixes $0$ and maps
every $[(i-1)/k, i/k]$ onto $I$.

\newcommand{\thmMainA}{
 For every non-degenerate totally regular continuum $X$ and every $a,b\in X$ we can find a convex metric
 $d=d_{X,a,b}$ on $X$ and Lipschitz surjections
 $\varphi_{X,a,b}:I\to X$, $\psi_{X,a,b}:X\to I$
 with the following properties:
\begin{enumerate}
  \item[(a)] $\lengthd{d}{X}= 1$;
  \item[(b)] the system
    $\CCc=\CCc_{X,a,b}=\{\varphi_{X,a,b}(J):\ J\text{ is a closed subinterval of } I\}$ 
    is a dense system of subcontinua of $X$;
  \item[(c)] for every $\varrho>1$ there are
  a constant $L_\varrho$ (depending only on $\varrho$) and
  $(\varrho,L_\varrho)$-LEL maps
   $$
    \varphi:(I,d_I,\CCc_I)\to (X,d,\CCc)
    \quad\text{and}\quad
    \psi:(X,d,\CCc)\to (I,d_I,\CCc_I)
   $$
    with $\varphi(0)=a$, $\varphi(1)=b$ , $\psi(a)=0$ and such that
    $\varphi=\varphi_{X,a,b}\circ f_k$, $\psi=f_l\circ \psi_{X,a,b}$ for some
    $k,l\ge 3$.
\end{enumerate}
Moreover, if $\Cut_X(a,b)$ is uncountable, $d,\varphi,\psi$ can be assumed to satisfy
\begin{enumerate}
  \item[(d)] $d(a,b)>1/2$ and $\psi(b)=1$.
\end{enumerate}
}

\begin{proposition}[\cite{SpA}]\label{T:MAINa}
\thmMainA{}
\end{proposition}

\newcommand{\thmMain}{
Keeping the notation from Proposition~\ref{T:MAINa},
for every $\varrho>1$, every non-degenerate totally regular continua $X,X'$ and every points
$a,b\in X$, $a',b'\in X'$
there are a constant $L_\varrho$ (depending only on $\varrho$) and $(\varrho,L_\varrho)$-LEL map
$$f:(X,d_{X,a,b},\CCc_{X,a,b})\to (X',d_{X',a',b'},\CCc_{X',a',b'})$$
with $f(a)=a'$ and, provided $\Cut_X(a,b)$ is uncountable, $f(b)=b'$.
Moreover, $f$ is equal to the composition $\varphi\circ\psi$ of two LEL-maps
$\psi:X\to I$ and $\varphi:I\to X'$.
}

\begin{proposition}[\cite{SpA}]\label{T:MAIN}
\thmMain{}
\end{proposition}

We will need Proposition~\ref{T:lipschitzXtoX}, which is a simple generalization
of Proposition~\ref{T:MAIN}.
Before giving the formulation of it we need to introduce the following notation.
If $p\ge 1$, $X_i$ ($i=1,\dots,p$) are non-degenerate subcontinua of $X$ and $a_i$ ($i=0,\dots,p$) are points of $X$ such that
$a_0\in X_1, a_p\in X_p$, $a_i\in X_i\cap X_{i+1}$ for $1\le i<p$
and $\CCc_i=\CCc_{X_i,a_{i-1},a_i}$ for $1\le i \le p$,
we write
$$
 (X,d,\CCc) = \coprod_{i=1}^p (X_i,d_i,\CCc_i)
$$
if $X=\bigcup_i X_i$, $d$ is a convex metric on $X$ such that
$d|_{X_i\times X_i}=d_{X_i,a_{i-1},a_i}$ for $i=1,\dots,p$
and $\CCc$ is the system of all unions
$C_i\cup X_{i+1}\cup \dots\cup X_{j-1}\cup C_j$
where $i\le j$, $C_i\in\CCc_{i}$, $C_j\in\CCc_{j}$,
 and, if $i<j$, $C_i=\varphi_{X_i,a_{i-1},a_i}([s,1])$ and
 $C_j=\varphi_{X_j,a_{j-1},a_j}([0,t])$ for some $s,t\in\III$.
Particularly, if $p=1$ then $d=d_1$ and $\CCc=\CCc_{X,a_0,a_1}$.

\begin{proposition}\label{T:lipschitzXtoX}
 Let $X,X'$ be non-degenerate totally regular continua.
 Let $a,b\in X$ and $d=d_{X,a,b}$, $\CCc=\CCc_{X,a,b}$.
 Let $a'_0, a'_1, \dots, a'_p$ be points of $X'$ such that
 $$
  (X',d',\CCc') = \coprod_{i=1}^p (X_i',d_{X_i',a_{i-1}',a_i'},\CCc_{X_i',a_{i-1}',a_i'}).
 $$
 Then for every $\varrho>1$ there is a map
 $f:X\to X'$ such that:
\begin{enumerate}
  \item[(a1)] $f:(X,d)\to (X', d')$ is a Lipschitz-$\tilde{L}$ surjection,
    where $\tilde{L}=\tilde{L}_{p,\varrho}$ depends only on $p$ and $\varrho$;
  \item[(a2)] for every $C\in\CCc$ we have
      $f(C)\in \CCc'$ and
     $$
      \text{if} \quad
      f(C)\not\supseteq X'_i \quad\text{for every }i
      \qquad     \text{then}    \quad
      \lengthd{d'}{f(C)}\ge\rho\cdot\lengthd{d}{C};
     $$
  \item[(b)] $f$ is the composition $\varphi\circ \psi$ of two Lipschitz-$\tilde{L}$ surjections
      $\psi:(X,d)\to \III$ and $\varphi:\III\to (X', d')$
      such that $\psi(a)=0$, $\varphi(0)=a_0'$, $\varphi(1)=a_p'$ and,
      provided $\Cut_X(a,b)$ is uncountable, $\psi(b)=1$.
\end{enumerate}
Notice that, by (b), $f(a)=a'_0$ and, provided $\Cut_X(a,b)$ is uncountable, $f(b)=a'_p$;
\end{proposition}

\begin{proof}
We may assume that $\varrho\ge 2$.
For $i=1,\dots,p$ put $d_i'=d'_{X_i',a_{i-1}',a_i'}$ and $\CCc_i'=\CCc'_{X_i',a_{i-1}',a_i'}$.
Let $L_\varrho$ be as in Proposition~\ref{T:MAIN} and let $\psi:(X,d,\CCc)\to (I,d_I,\CCc_I)$,
$\varphi_i'=\varphi_{X_i',a_{i-1}',a_i'}\circ f_{k_i}:(I,d_I,\CCc_I)\to (X_i',d_i',\CCc_i')$
($i=1,\dots,p$) be $(\varrho,L_\varrho)$-LEL maps
with $\varphi_i'(0)=a_{i-1}'$, $\varphi_i'(1)=a_{i}'$, $\psi(a)=0$ and, provided
$\Cut_X(a,b)$ is uncountable, also $\psi(b)=1$.

For $i=1,\dots,p$ let $I_i = [ (i-1)/p,\, i/p \,]$ and
$g_i:I_i\to I$ be the increasing linear surjection.
Define $f:X\to X'$ by
$f = \varphi' \circ \psi$, where
$\varphi':I\to X'$ is given by $\varphi'|_{I_i}=\varphi'_i \circ g_i$ for $i=1,\dots,p$.
Then $\varphi'(i/p)=a_i'$ for $0\le i\le p$, $\varphi'$ is continuous and Lipschitz with $\Lip(\varphi')\le pL_\varrho$
by Lemma~\ref{L:pwLip}.
Since $\Lip(f) \le p L_\varrho^2$, $f$ satisfies (a1) and (b) with $\tilde{L}=p L_\varrho^2$.

 To show that $f$ satisfies also (a2) fix any $C\in\CCc$ and put $L=\psi(C)$
 and $L_i=L\cap I_i$ for $i=1,\dots,p$;
 notice that $L$ and nonempty $L_i$'s are closed intervals.
 The fact that $f(C)=\varphi'(L)\in \CCc'$ immediately follows from
 the definitions of $f$ and $\CCc'$.
 If $L$ contains some $I_i$ then $f(C) \supseteq X_i'$. Otherwise
 $\abs{L}\ge\varrho\cdot \lengthd{d}{C}$ since $L\ne I$ and
 there
 is $i$ such that $L=L_i\cup L_{i+1}$. We may assume that $\abs{L_i}\ge \abs{L}/2$.
 Then
 $$
  \lengthd{d'}{f(C)}
  \ge \lengthd{d'}{\varphi'_i \circ g_i (L_i)}
  \ge \varrho p \cdot \abs{L_i}
  \ge \frac{\varrho^2 p}{2}\cdot \lengthd{d}{C}
  \ge \varrho\cdot\lengthd{d}{C}
 $$
 since $\varrho\ge 2$. So $f$ satisfies also (a2) and the proof is finished.
\end{proof}


\section{$P$-Lipschitz maps and an entropy bound}\label{S:entropyBound}
In this section we introduce a special class of maps called
$P$-Lipschitz maps (do not confuse with Lipschitz-$L$ maps)
and we show a relatively easy way of obtaining an
upper bound for their entropy (see
Proposition~\ref{P:entropyOfPLipschitz}), which will be used in Section~\ref{S:applications}.
For the definition of a splitting see Section~\ref{SS:continua}.

\begin{definition}[$P$-Lipschitz maps]\label{D:generalizedMarkovMap}
 Let $X$ be a non-degenerate continuum and $f:X\to X$ be a continuous map. Let $P$ be a finite $f$-invariant
 subset of $X$, $\AAa$ be a splitting of $X$ with $P_\AAa\subseteq P$
 and $(L_A)_{A\in\AAa}$ be positive constants.
 Then we say that $f$ is a \emph{$P$-Lipschitz map (w.r.t.~the splitting $\AAa$ and the constants $(L_A)_{A\in\AAa}$)}
 if, for every $A\in\AAa$, $f(A)$ is non-degenerate and $\Lip(f|_A)\le L_A$.
\end{definition}
Roughly speaking, $P$-Lipschitz maps are ``piecewise-Lipschitz''
maps with the ``pieces'' being subcontinua intersecting only in a
finite invariant subset $P$. Trivially, every non-constant Lipschitz
map is $P$-Lipschitz with $P=\emptyset$ and $\AAa=\{X\}$. On the
other hand, if the metric $d$ of $X$ is convex then every
$P$-Lipschitz map is Lipschitz with $\Lip(f)\le\max_{A\in\AAa} L_A$
(see Lemma~\ref{L:pwLip}).

Let $f$ be a $P$-Lipschitz map w.r.t.~$\AAa$.
For $A,B\in\AAa$ we write $A\to B$ or, more
precisely, $A\toPar{f} B$ provided $f(A)$ intersects the interior of
$B$. So if $A\to B$ then $f(A)\cap B$ is a finite union of
subcontinua of $X$ at least one of which is non-degenerate; the
number of components of $f(A)\cap B$ is smaller than or equal to the
cardinality of $\partial B$. Since every $f(A)$ is
non-degenerate, for every $A\in\AAa$ there is $B\in\AAa$ with $A\to
B$ and, provided $f$ is surjective, also $C\in\AAa$ with $C\to A$.

The \emph{$P$-transition graph $G_f$ of $f$} is the directed
graph the vertices of which
are the sets $A\in\AAa$ and the edges of which correspond to $A\to B$.
The corresponding $01$-transition matrix will be called the
\emph{$P$-transition matrix} and
will be denoted by $M_f$.
For an integer $n\ge 1$ put
$$
 \AAa^n = \{\AAA=(A_0,A_1,\dots, A_{n-1}):\
 A_i\in \AAa,\ A_0\to A_1\to \dots\to A_{n-1}
 \}.
$$
So $\AAa^n$ is the set of all paths of length $n-1$ in $G_f$.

\begin{lemma}\label{L:MarkovMapTrajectories}
 Let $f$ be a $P$-Lipschitz map w.r.t.~$\AAa$ and  let $x\in X$.
 Then for every integer $n\ge 1$ there is $\AAA=(A_0,A_1,\dots, A_{n-1})\in\AAa^n$ such that
 $$
  f^i(x)\in A_i
  \qquad
  \text{for every } i=0,1,\dots, n-1.
 $$
\end{lemma}
\begin{proof}
 We prove the lemma by induction. If $n=1$ then the assertion follows from the fact
 that $\AAa=\AAa^1$ covers $X$. Assume now that the assertion of the lemma is true for some $n\ge 1$;
 we are going to show that it is true also for $n+1$.
 To this end take any point $x\in X$; we want to find $\AAA\in\AAa^{n+1}$ with
 $f^i(x)\in A_i$ for $i\le n$. By the induction hypothesis there is
 $(A_0,A_1,\dots, A_{n-1})\in\AAa^n$ such that $f^i(x)\in A_i$ for every $i\le n-1$.
 Put $y=f^{n-1}(x)\in A_{n-1}$. If $f(y)\not\in P$ then $f(y)\in\interior{A_n}$ for some $A_n\in\AAa$;
 hence $A_{n-1}\to A_n$ and we can put
  $\AAA=(A_0,A_1,\dots, A_{n-1},A_n)\in\AAa^{n+1}$.
 Now assume that $f(y)\in P$. Let $B_1,\dots, B_{k}$ ($k\ge 1$) be the collection
 of all sets from $\AAa$ containing $f(y)$.
 Since $f(A_{n-1})$ has no isolated point (it is a non-degenerate continuum)
 and contains $f(y)$, it must intersect the interior of some $B_i$. If we put
 $\AAA=(A_0,A_1,\dots, A_{n-1},B_i)$, we are done.
\end{proof}

Let $X$ be a continuum with a convex metric $d$ and let
$f$ be a $P$-Lipschitz map on $X$ w.r.t.~$\AAa$ and
$(L_A)_{A\in\AAa}$. For every nonempty subset $\BBb$ of $\AAa$ put
$$
 L_{\BBb} = \max_{A\in\BBb} L_A.
$$
Since any $P$-Lipschitz map $f$ has $\Lip(f)\le L_{\AAa}$, we immediately
have (see e.g.~\cite[Theorem~3.2.9]{Katok})
$$
 h(f) \le \dimUpper(X)\cdot \log^+ L_\AAa
$$
where $\log^+ L=\max\{\log L, 0\}$.

Assume now that
$X$ is a non-degenerate totally regular continuum and the metric $d$ of $X$
is convex with $\lengthd{d}{X}<\infty$.
Then $\dimUpper(X)=1$ (see Lemma~\ref{L:separatedSetsInContinuaOfFiniteLength}),
hence $h(f)\le \log^+ L_\AAa$.
But this upper bound is often too pessimistic. Consider e.g.~the following
example. Let $f$ and $\AAa=\{X_0,\dots,X_{k-1}\}$ be such that
$f(X_i)\subseteq X_{(i+1) \operatorname{mod} k}$ and that
$L_{X_0}=L>1$, $L_{X_i}\le 1$ for $i\ge 1$. Then $L_\AAa=L$ and we have
$h(f)\le \log L$. But since $X_i$'s are $f^k$-invariant and $\Lip(f^k|_{X_i})\le L$,
we have a much better estimate $h(f)\le (1/k) \log L$.
The key point here is the fact that the sets $A\in\AAa$ with
``large'' $L_A$ occurs ``rarely'' in paths $\AAA\in\AAa^n$ of $f$.

To formalize this idea, for
nonempty $\BBb\subseteq \AAa$ put
\begin{equation}\label{EQ:thetaB}
  \theta_\BBb
  =\limsup\limits_{n\to\infty} \frac{k_n^{\BBb}}{n}
  \qquad
  \text{where}
  \quad
  k^{\BBb}_n
  =\max\limits_{\AAA\in \AAa^{n}} \card{ \{k=0,\dots, n-1: A_k\not\in\BBb}\} \,.
\end{equation}
This quantity measures the maximal ``asymptotic frequency'' of occurrences of
$A\in \AAa\setminus\BBb$ in paths of $f$. The following proposition
asserts that if $L_\BBb$ is close to $1$ and $\theta_\BBb$ is sufficiently small,
then the entropy of $f$ can be small even for  large $L_\AAa$.

\begin{proposition}[Upper bound for the entropy of $P$-Lipschitz maps]\label{P:entropyOfPLipschitz}
 Let $X$ be a non-degenerate totally regular continuum endowed with a convex metric $d$
 such that $\lengthd{d}{X}<\infty$.
 Let $f:X\to X$ be $P$-Lipschitz w.r.t.~$\AAa$ and $(L_A)_{A\in\AAa}$.
 Then, for every nonempty subsystem $\BBb$ of $\AAa$,
 $$
  h(f) \le \log^+ L_\BBb + 2\theta_\BBb \log^+ L_\AAa.
 $$
\end{proposition}
\begin{proof}
 Since the assertion of the proposition does not change if we replace every $L_A$
 by $\max\{L_A,1\}$, we may assume that $L_A\ge 1$ for every $A\in\AAa$.
 Fix any $0<\eps < d(X)$ such that
  $\eps<\dist(A,B)$ for every disjoint $A,B\in\AAa$
  and $\eps< d(x,y)$ for every distinct $x,y\in P$.

 Let $n\in\NNN$ and let $S$ be an $(n,\eps)$-separated subset of $X$ of maximal cardinality
 $\card{S}=s_n(X,\eps,f)$.
 Take any $x\ne y$ from $S$. By Lemma~\ref{L:MarkovMapTrajectories}
 there are $\AAA^x=(A_i^x)_{i<n}, \AAA^y=(A_i^y)_{i<n}\in\AAa^n$
 such that $x_i = f^i(x)\in A_i^x$ and $y_i = f^i(y)\in A_i^y$ for $i<n$.
 Put $d_i=d(x_i,y_i)$ for $i<n$ and
 $j = \min\{i: d_i>\eps\}$; then $0\le j\le n-1$.

 Fix any $0\le i< j$. Let $B$ be a geodesic arc from $x_i$ to $y_i$.
 Then, by (\ref{EQ:length-diam}), $d(B)\le \length{B}=d(x_i,y_i)\le \eps$. By the
 choice of $\eps$ the set $B$ contains at most one point from $P$.
 If $B\cap P=\emptyset$ then $B$ does not intersect the boundaries of $A_i^x$, $A_i^y$
 and so $B\subseteq \interior{A_i^x} \cap \interior{A_i^y}$.
 Hence $A_i^x = A_i^y$ and
 \begin{equation}\label{EQ:PLip1}
  d_{i+1}\le L_{A_i^x} \cdot d_i.
 \end{equation}
 Otherwise $B\cap P=\{z\}$ is a singleton. We claim that $z\in A_i^x$. Indeed,
 if $z\ne x_i$ then
 $x_i$ belongs to the interior of $A_i^x$.
 Since $B$ is connected, intersects the interior of ${A_i^x}$
 and $\partial A_i^x\subseteq P$,
 the subarc $x_i z$ of $B$ is contained
 in $A_i^x$ and so $z\in A_i^x$.
 If $z=x_i$ then again $z\in A_i^x$.
 Analogously $z\in A_i^y$ and thus $z\in A_i^x\cap A_i^y$.
 Using the facts that $f|_A$ is Lipschitz on every $A\in\AAa$
 and that $B$ is geodesic (hence $d(x_i,z)+d(z,y_i)=d_i$)
 we obtain
 \begin{equation}\label{EQ:PLip2}
 \begin{split}
  d_{i+1}
  &\le d(x_{i+1},f(z)) + d(f(z), y_{i+1})
  \le L_{A_i^x} \cdot d(x_i,z) +  L_{A_i^y} \cdot d(z,y_i)
 \\
  &\le \max\{L_{A_i^x}, L_{A_i^y}\} \cdot d_i \,.
 \end{split}
 \end{equation}
 Combination of (\ref{EQ:PLip1}) and (\ref{EQ:PLip2}) for $i<j$
 gives
 $$
  d_{j}
  \le d_0 \cdot \prod\limits_{i<j} \max\{L_{A_i^x}, L_{A_i^y}\}
  \le d_0 \cdot \prod\limits_{i<n} \max\{L_{A_i^x}, L_{A_i^y}\} \,.
 $$
 Thus, by the definition (\ref{EQ:thetaB}) of $k_n=k_n^\BBb$,
 $$
  \eps
  < d_{j}
  \le d_0 \cdot L_{\AAa}^{2k_n} \cdot L_{\BBb}^{n-2k_n}
  \le d_0 \cdot L_{\AAa}^{2k_n} \cdot L_{\BBb}^{n} \,.
 $$
 So the set $S$ is $\eps'$-separated with $\eps'=\eps/ (L_{\AAa}^{2k_n} \cdot L_{\BBb}^{n})$
 and Lemma~\ref{L:separatedSetsInContinuaOfFiniteLength} gives
 $$
  s_n(X,\eps,f)
  \le
  s(X,\eps')
  \le \frac{2 \length{X}}{\eps} \cdot L_{\AAa}^{2k_n} \cdot L_{\BBb}^{n} \,.
 $$
 From this and (\ref{EQ:BowenDefEntropy}),
 $h(f) \le 2\theta_\BBb \log L_\AAa + \log L_\BBb$.
\end{proof}

In the following example we show how Proposition~\ref{P:entropyOfPLipschitz}
can be used to obtain a ``good'' upper estimate of the topological entropy.

\begin{example}
Consider a graph $X$ which is not a tree. We can write $X=I\cup G$, where
$I$ is a free arc in $X$ identified with $[0,1]$ and $G$ is a subgraph of $X$
such that $I\cap G=\{0,1\}$. Let $n\ge 3$ and $g=g_n:X\to X$ be the map constructed
in the proof of \cite[Theorem~4.1]{ARR}. Recall that
if we put $X_i=\left[i/n, (i+1)/n \right]$ for $0\le i<n$, $i\ne 1$ and
$X_1^-=\left[ \frac{1}{n}, \frac{3}{2n} \right]$,
$X_1^+=\left[ \frac{3}{2n}, \frac{2}{n} \right]$,
then $g$ maps linearly $X_0$ onto $X_1^-\cup X_1^+$, $X_1^-$ onto $X_1^+$,
$X_1^+$ onto $X_1^+\cup X_2$ and $X_i$ onto $X_{i+1}$ for $i=2,\dots,n-2$.
Moreover, $g|_{X_{n-1}}:X_{n-1}\to G$ and $g|_G:G\to X_0$ are piecewise linear
maps with the number of pieces not depending on $n$.

Put $\lambda=\sqrt[n-2]{2}$ and
define the convex metric $d=d_n$ on $X$ in such a way that
$d(0,1)=d\left( 0, \frac{1}{n} \right)=1$,
$d\left( \frac{1}{n}, \frac{3}{2n} \right)=d\left( \frac{3}{2n}, \frac{2}{n} \right)
=\frac 12$,
$d\left( \frac{i}{n}, \frac{i+1}{n} \right)=\frac 12 \lambda^{i-1}$
for $i=2,\dots,n-1$, and
the Lipschitz constants of $g|_{X_{n-1}}$, $g|_G$ are bounded from above by a constant
$L$ not depending on $n$.
Then $g$ is $P$-Lipschitz with
$P=\left\{ \frac{i}{n},\ i=0,\dots,n\right\} \cup
\left\{ \frac{3}{2n}\right\}$,
$\AAa=\{X_0,X_1^-,X_1^+,X_2,\dots,X_{n-1},G\}$
and $L_{X_0}=L_{X_1^-}=1$, $L_{X_1^+}=L_{X_2}=\dots=L_{X_{n-2}}=\lambda$,
$L_{X_{n-1}}=L_G=L$.
Put $\BBb=\AAa\setminus \{G,X_{n-1}\}$.
Then
$L_\AAa=L$, $L_\BBb=\lambda$ and
an easy computation gives that $\theta_\BBb\le\frac{2}{n}$.
According to Proposition~\ref{P:entropyOfPLipschitz} we have
$$
 h(g_n)
 \le \log \lambda + \frac {2}{n} \log L
 \le \frac{\log(2L^2)}{n-2} \,,
$$
and thus $\limsup_n h(g_n)=0$.
\end{example}

We finish this section with a simple lemma giving a
sufficient condition for transitivity and exactness
of $P$-Lipschitz maps. Recall that a $01$-square matrix $M$
is called \emph{irreducible} if for every indices $i,j$ there is $n>0$
with $M^n_{ij}>0$; i.e.~the corresponding directed graph
$G_M$ is \emph{strongly connected} (there is a path from each vertex to every
other vertex).
The \emph{period $p_M$} of an irreducible matrix $M$
is the greatest common divisor of those $n>0$ for which $M^n_{ii}>0$ for some $i$.
Equivalently, $p_M$ is the greatest common divisor of the lengths of cycles
(closed paths) in $G_M$. An irreducible matrix $M$ and the corresponding graph $G_M$
are called \emph{primitive} provided $p_M=1$.

\begin{lemma}\label{L:exactnessOfPLipschitz}
 Let $X$ be a non-degenerate totally regular continuum endowed with a convex metric $d$
 such that $\lengthd{d}{X}<\infty$.
 Let a map $f:X\to X$ be $P$-Lipschitz w.r.t.~$\AAa$ such that
 $A\to B$ implies $f(A)\supseteq B$ for $A,B\in\AAa$.
 Assume that there is a dense system $\DDd$ of subcontinua of $X$ such that
 for every $D\in\DDd$ there
 are $n\in\NNN$, $A\in\AAa$ with $f^n(D)\supseteq A$.
 Then the following hold:
 \begin{enumerate}
    \item[(a)] $f$ is transitive but not totally transitive
      if and only if $M_f$ is irreducible but not
      primitive;
    \item[(b)] $f$ is exact if and only if $M_f$ is primitive.
 \end{enumerate}
\end{lemma}

\begin{proof}
If $M_f$ is not irreducible there are $A,B\in\AAa$ such that there is no path from
$A$ to $B$. Hence $f^n(\interior{A})\cap \interior{B} = \emptyset$ for every $n>0$,
which contradicts transitivity of $f$. If $M_f$ is irreducible with the period $p>1$,
then there are $A,B\in\AAa$ such that there is no path from
$A$ to $B$ of length $n p$. I.e.~$f^p$ is not transitive and thus $f$ is not
totally transitive.

Assume now that $M_f$ is irreducible and take any nonempty open sets $U,V$.
Since $\bigcup_{A\in \AAa} \interior{A}\supseteq X\setminus P$
is dense in $X$, there is $B\in\AAa$ intersecting $V$.
By the assumption on $\DDd$ there are $D\in \DDd$, $n\in \NNN$ and $A\in\AAa$
such that $D\subseteq U$ and $f^{n}(D)\supseteq A$. Since $M_f$ is irreducible
there is $m$ with $f^{m}(A)\supseteq B$. Hence $f^{n+m}(D)\supseteq B$ and
so $f^{n+m}(U)$ intersects $V$. Thus $f$ is transitive.
Moreover, if $M_f$ is primitive there is $p$ such that, for every
$C\in\AAa$, there is a path
of length $p$ from $A$ to $C$. So $f^{n+p}(U)\supseteq
f^{n+p}(D)\supseteq f^{p}(A)=X$ and $f$ is exact.
\end{proof}

\section{Proof of \theoremMainRef{}}\label{S:applications}
In this section we prove \theoremMainRef{}. We use the notation from Section~\ref{SS:tentLike}.

\subsection{Lipschitz, exactly Devaney chaotic maps}

\begin{lemma}\label{P:exactLipschitz}
 There is $L>0$ such that
 any non-degenerate totally regular continuum  $X$ admits a
 convex metric $d$ with $\lengthd{d}{X}<\infty$
 and Lipschitz-$L$,
 Devaney chaotic maps $f,g:(X,d)\to (X,d)$ with positive entropies
 such that $f$ is not totally transitive
 and $g$ is exact.
\end{lemma}
\begin{proof} The existence of a map $g:(X,d)\to (X,d)$ was shown in \cite{SpA},
 so we only need to construct $f$. If $X$ contains a circle, there are points $a\ne b$
 such that $X$ is the union of two subcontinua $X_0,X_1$
 with $X_0\cap X_1=\{a,b\}$ and $\Cut_{X_i}(a,b)$ is uncountable for $i=0,1$.
 (Such points $a,b$ can be chosen as follows.
 By \cite{BNT} we can write
 $X=\invlim (X_n,f_n)$, where $X_n$ are graphs and $f_n:X_{n+1}\to X_n$ are monotone;
 let $\pi_n:X\to X_n$ ($n\in\NNN$) be the natural projections.
 Since $X$ is not a dendrite,
 by \cite[Theorem~10.36]{Nad}
 there is $m$ such that $X_m$ contains a circle $S$. Let $C$ be the set of all points
 $x_m\in S$ which are not vertices of $X_m$ and are such that $\pi_m^{-1}(x_m)$
 are singletons. Then $S\setminus C$ is countable
 since every system of disjoint non-degenerate subcontinua of $X$ is countable.
 Take $a_m\ne b_m$ from $C$ and put $a=\pi_m^{-1}(a_m)$, $b=\pi_m^{-1}(b_m)$.
 Write $X_m$ as the union $X_{m,0}\cup X_{m,1}$
 of two non-degenerate subgraphs with $X_{m,0}\cap X_{m,1}=\{a_m,b_m\}$.
 Then $X_i=\pi_m^{-1}(X_{m,i})$ ($i=0,1$), being inverse limits of continua
 \cite[Proposition~2.1.17]{Macias},
  are non-degenerate subcontinua of $X$
 covering $X$ and $X_0\cap X_1=\pi_m^{-1}(X_{m,0}\cap X_{m,1}) = \{a,b\}$.)

 Without loss of generality we may assume that
 $d_{X_0,a,b}(a,b)=d_{X_1,a,b}(a,b)$ (otherwise we replace one of the metrics
 by a constant multiple of it).
 Let $d$ be the convex metric on $X$ defined by (\ref{EQ:convexMetricOnUnion})
 for the splitting $\{X_0,X_1\}$ and metrics $d_0=d_{X_0,a,b}$, $d_1=d_{X_1,a,b}$;
 by Lemma~\ref{L:convexMetricOnUnion}, $d$
 coincides with $d_{X_i,a,b}$ on $X_i$
 for $i=0,1$.
 Let $f_0:X_0\to X_1$, $f_1:X_1\to X_0$ be maps
 from Proposition~\ref{T:lipschitzXtoX} (with $\varrho>1$, $p=1$) fixing points $a$ and $b$.
 Define $f:X\to X$ by $f(x)=f_i(x)$ if $x\in X_i$.
 Then, for $i=0,1$, $f^2|_{X_i}:X_i\to X_i$ is LEL and so, by
 Proposition~\ref{P:tentLikeIsExact},
 $f^2|_{X_i}$ is
 exactly Devaney chaotic.
 Thus $f$ is Devaney chaotic and has positive entropy.
 Since $f^2$ is not transitive, $f$ is not totally transitive.

 If $X$ does not contain a circle, then $X$ is a dendrite. Take any non-end point
 $a\in X$ and write $X=X_0\cup X_1$, where $X_i$'s are non-degenerate
 subdendrites with $X_0\cap X_1=\{a\}$. Fix $b_i\in X_i\setminus\{a\}$ and notice that,
 since $X_i$'s are dendrites, the sets
 $\Cut_{X_i}(a,b_i)$ are uncountable.
 Now we can proceed analogously as in the non-dendrite case; the only change is that
 instead of $f_i(b)=b$ we require $f_0(b_0)=b_1$ and $f_1(b_1)=b_0$.
\end{proof}

\begin{corollary}\label{C:IT(X)_is_finite}
 Every non-degenerate totally regular continuum admits a (not totally) Devaney
 chaotic, as well as an exactly Devaney chaotic map with finite positive entropy.
 Every compact metric space which is the disjoint union of finitely many non-degenerate
 totally regular continua admits a Devaney chaotic map with finite positive entropy.
\end{corollary}

\subsection{Exactly Devaney chaotic maps with arbitrarily small entropy}
\label{SS:exactSmallEntropy}
In what follows we show that under some conditions a totally regular continuum $X$
admits (exactly) Devaney chaotic maps with arbitrarily small entropy.
We consider two subclasses of totally regular continua:
those containing arbitrarily large generalized stars and
those containing a non-disconnecting
free arc. The constructions are modifications
of those for stars \cite[Theorem~1.2]{AKLS} and for graphs which are not trees
\cite[Theorems~3.7 and 4.1]{ARR}. The key difference is that instead of constructing
piecewise linear Markov maps we construct $P$-Lipschitz maps using
Proposition~\ref{T:lipschitzXtoX}.

Let $X$ be a continuum and $k\ge 2$ be an integer.
We say that $X$ \emph{contains a generalized $k$-star}
if there are a point $a\in X$ and $k$ components $C_1,\dots,C_k$
of $X\setminus\{a\}$ such that
the closures $X_i=\closure{C_i}=C_i\cup\{a\}$ ($i=1,\dots,k$) are homeomorphic relative to $a$;
i.e.~for every $i,j$ there is a homeomorphism $h_{ij}:X_i\to X_j$ fixing the point $a$
(see \FigureKStar{}).
Recall that $X$ \emph{contains arbitrarily large generalized stars}
if it contains a generalized $k$-star for every $k\ge 2$.

\begin{lemma}\label{T:I(omega-star)}
There is a constant $c>0$ such that
any totally regular continuum $X$ containing a generalized $k$-star
($k\ge 2$) admits a (not totally) Devaney chaotic map $f$
as well as an exactly Devaney chaotic map $g$ with positive topological entropies
bounded from above by $c/k$.
Consequently, if $X$ contains arbitrarily large generalized stars then $\IED(X)=0$.
\end{lemma}

\begin{figure}[ht!]
  \includegraphics{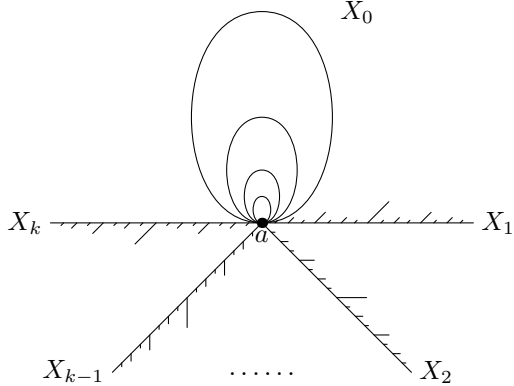}
  \caption{A continuum containing a generalized $k$-star}
  \label{Fig:kstar}
\end{figure}

\begin{proof}
Since $X$ contains a generalized $k$-star, there is $a\in X$ such that we can write
$
 X = X_0\cup X_1\cup\dots \cup X_{k},
$
where $X_i$'s are subcontinua of $X$,
$X_i\cap X_j=\{a\}$ for every $i\ne j$  and
$X_i,X_j$ are homeomorphic relative to $a$ for $i,j\ge 1$.
Further, $X_1,\dots,X_k$ are non-degenerate and, by replacing
$k$ with $k-1$ if necessary, we may assume that also $X_0$ is such.
For $i=0,\dots,k$ put $d_i=d_{X_i,a,a}$ and $\CCc_i=\CCc_{X_i,a,a}$.

We may assume that the metrics $d_i$ are such that
for every $1\le i<k$ there is an isometry $f_i:(X_i,d_i)\to (X_{i+1},d_{i+1})$
fixing $a$. Since $\bigcup_{i\ne j} X_i\cap X_j=\{a\}$ is a singleton,
by Lemma~\ref{L:convexMetricOnUnion} there is a convex metric $d$ on $X$ such that $d|_{X_i\times X_i}=d_i$
for every $i$. Let $f_{k}:X_{k}\to X_0$, $f_0:X_0\to X_1$ be maps
from Proposition~\ref{T:lipschitzXtoX} (with $p=1$, $\varrho>1$)
such that $f_{k}(a)=f_0(a)=a$. Without loss of generality we may assume
that if $C\in \CCc_{i}$ ($i=1,\dots,k-1$) then $f_i(C)\in \CCc_{i+1}$.
We define the map $f:X\to X$ by
 $f(x)=f_i(x)$ for $x\in X_i$, $i=0,\dots,k$.
 Proposition~\ref{P:tentLikeIsExact} implies that
 $f^{k+1}|_{X_0}: X_0\to X_0$, being LEL, is exactly Devaney chaotic.
 Hence $h(f)>0$, $f$ is a (not totally) Devaney chaotic map,
 and, since $d$ is convex, $f$ is Lipschitz with $\Lip(f)\le\tilde{L}$
 by Lemma~\ref{L:pwLip}.
 Moreover, since
 $\Lip(f^{k+1}) \le \tilde{L}^2$ we have $h(f)\le \tilde{L}^2/(k+1)\le c/k$
 for some constant $c$ depending only on $\tilde{L}$.

 To construct an exactly Devaney chaotic map $g:X\to X$
 we take the metric $d$ on $X$ as in the previous case
 and we define $g|_{X_i}=f_i$ for $i=1,\dots,k$. The only difference
 is the definition of $g$ on $X_0$: we put $g|_{X_0}=g_0$,
 where
 $g_0:X_0\to (X_1\cup X_2)$ is a map obtained
 from Proposition~\ref{T:lipschitzXtoX}
 with $p=2$, $\varrho'>2$, $X'=X_1\cup X_2$ and $a'_i=a$ for $i=0,1,2$.
 Then $g_0(a)=a$.
 By Proposition~\ref{T:lipschitzXtoX} we have that,
 for $C\in\CCc_{0}$,
 if $g_0(C)$ contains
 neither $X_1$ nor $X_2$, then
 ${g_0(C)}=C_1\cup C_2$, where $C_1\in\CCc_{1}$ and $C_2\in\CCc_{2}$
  are continua and
  $\length{g_0(C)}=\length{C_1}+\length{C_2}\ge \varrho' \cdot\length{C}$.
  Hence
\begin{equation}\label{EQ:kstar1}
 \length{C'}
     \ge
 \frac{\varrho'}{2}\cdot \length{C}
 \qquad
 \text{for some}\quad
 C'\in \CCc_{1}\cup \CCc_{2},\
 C'\subseteq g(C).
\end{equation}
 Notice that $\Lip(g)$
 is bounded from above by an absolute constant $L$,
 see Lemma~\ref{L:pwLip}.

 We are going to show the entropy bound for $g$ using
 Proposition~\ref{P:entropyOfPLipschitz}.
 Put $P=\{a\}$,
 $\AAa=\{X_0,\dots,X_k\}$, $\BBb=\{X_1,\dots,X_{k-1}\}$,
 $L_{X_i}=1$ for $1\le i<k$ and
 $L_{X_0}=L_{X_k}=L$. Then $g$ is $P$-Lipschitz w.r.t.~$\AAa$ and $(L_{X_i})_i$, $L_\AAa=L$ and $L_\BBb=1$.
 Moreover, $X_i\to X_j$ if and only if $j=(i+1) \operatorname{mod} (k+1)$
 or $i=0$, $j=2$.
 Fix any $n\in\NNN$ and $\AAA=(A_0,\dots,A_{n-1})\in\AAa^n$
 with exactly $k_n=k_n^\BBb$ members from $\AAa\setminus\BBb$, see (\ref{EQ:thetaB}).
 Let $l_1<l_2<\dots<l_{k_n}$ be the indices $l$ with $A_l\not\in\BBb$,
 i.e.~with $A_l\in\{X_0,X_k\}$.
 Then, for any $l_j$, $1\le i\le k-1$ and $l_j+i<n$, it holds that
 $A_{l_j+i}\in\{X_{i-1},X_i,X_{i+1}\}$
 and thus $A_{l_j+2},\dots,A_{l_j+k-2}\in \BBb$.
 So $l_{j+2}-l_j\ge k-1$ for every $j\le k_n-2$.
 We see at once that
 $n-1\ge l_{k_n} \ge (k_n-1)(k-1)/2$,
 i.e.~$\theta_\BBb
  = \limsup_{n}    k_n/n
  \le
  2/(k-1)$.
 By Proposition~\ref{P:entropyOfPLipschitz},
 $h(g)\le \log^+ L_\BBb + 2\theta_\BBb \log^+ L_\AAa \le (4/(k-1)) \log L
 \le c/k$ for some absolute constant $c$.

 To prove that $g$ is exact we use Lemma~\ref{L:exactnessOfPLipschitz}
 with $\DDd = \cup_i \CCc_{i}$.
 To this end take any $i$ and $C\in \CCc_{i}$ and
 suppose that for every $n$ the set
 $g^n(C)$ does not contain any $X_i$.
 Then Proposition~\ref{T:lipschitzXtoX} and (\ref{EQ:kstar1})
 give that, for every $n\ge 1$,
 $$
  \length{g^n(C)} \ge q^n \cdot \length{C}
  \qquad\text{with}\quad
  q=\min\{\varrho,\varrho'/2\}>1.
 $$
 (The inequality can be shown by induction as follows.
 Trivially it holds for $n=0$.
 Put $C_0=C$ and assume that, for some $n\ge 0$, there is $i$ such that the set
 $g^n(C)$ contains some $C_n\in \CCc_{i}$ with
 $\length{C_n} \ge q^n \cdot\length{C}$.
 If $i\ne 0$ then put $C_{n+1}=g(C_n)\in \CCc_{(i+1) \operatorname{mod} \, (k+1)}$
 and use Proposition~\ref{T:lipschitzXtoX}.
 If $i=0$, put $C_{n+1}=C'$, where $C'$ is the continuum from (\ref{EQ:kstar1}).)
 But this contradicts the fact that $\length{X}<\infty$.
 Hence $\DDd$ satisfies the assumption from Lemma~\ref{L:exactnessOfPLipschitz}.
 Notice also that $A\to B$ implies $g(A)\supseteq B$.
 And since the $P$-transition graph of $g$
 contains cycles $X_0\to X_1\to\dots\to X_k\to X_0$ and
 $X_0\to X_2\to X_3\to\dots\to X_k\to X_0$ of lengths $k+1$ and $k$,
 the $P$-transition matrix of $g$ is primitive. Hence,
 by Lemma~\ref{L:exactnessOfPLipschitz}, $g$ is exact.

 To finish it suffices to show that $g$ has dense periodic points.
 Let $\tilde{X}=X\cup A$ be a totally regular continuum obtained
 from $X$ by adding to it an arc $A$ such that $A\cap X=\{a\}$. Take a convex metric
 $\tilde{d}$ on $\tilde{X}$ such that it coincides with $d$ on $X$ and
 the length of $A$ is $1$. Then there is an isometry from $\III$ onto $A$
 mapping $0$ to $a$. So, by the choice of $g_0$, there
 are Lipschitz maps $\psi:X_0\to A$ and $\varphi:A\to X_1\cup X_2$
 with $\varphi\circ\psi=g_0$ and $\psi(a)=\varphi(a)=a$.
 Define the map $\tilde{g}:\tilde{X}\to \tilde{X}$ by
 $$
  \tilde{g}({x}) = \begin{cases}
    g({x}) & \text{if } x\in X_1\cup \dots\cup X_k;
   \\
    \psi(x) & \text{if } x\in X_0;
   \\
    \varphi(x) & \text{if } x\in A.
  \end{cases}
 $$
 Analogously as before, $\tilde{g}$ is exact. Since $\tilde{X}$ contains
 a disconnecting interval (an open set homeomorphic to $(0,1)$ such that
 any point of it disconnects $X$ into exactly two components), $\tilde{g}$
 has dense periodic points by \cite[Theorem~1.1]{AKLS}.
 Trivially, any periodic point of $\tilde{g}$ in $X$ is also a periodic point of $g$.
 So also the periodic points of $g$ are dense.
\end{proof}

\begin{lemma}\label{T:I(free-arc)}
 There is $c>0$ with the following property.
 Let $X$ be a totally regular continuum
 which can be written as the union of non-degenerate subcontinua
 $$
  X = X_0\cup X_1\cup \dots\cup X_k,
 $$
 where $k\ge 2$ and the following hold:
\begin{enumerate}
    \item[(a)] there are distinct points $a_0,\dots,a_k$ with
     $X_i\cap X_{(i+1) \operatorname{mod} (k+1)} = \{a_i\}$ for
     $i=0,\dots,k$;
    \item[(b)] $X_i\cap X_j=\emptyset$ whenever $2\le \abs{i-j}<k$;
    \item[(c)] for $i=1,\dots,k-1$ there is a homeomorphism $h_i:X_i\to X_{i+1}$
     such that $h_i(a_{i-1})=a_{i}$ and $h_i(a_i)=a_{i+1}$;
    \item[(d)] $\Cut_{X_1}(a_0,a_1)$ is uncountable.
\end{enumerate}
Then $X$ admits a (not totally) Devaney chaotic map $f$
and an exactly Devaney chaotic map $g$ with positive topological entropies
bounded from above by $c/k$.
\end{lemma}

See \FigureFreeArc{} for an illustration of a continuum satisfying the assumptions
of the lemma.

\begin{figure}[ht!]
 \includegraphics{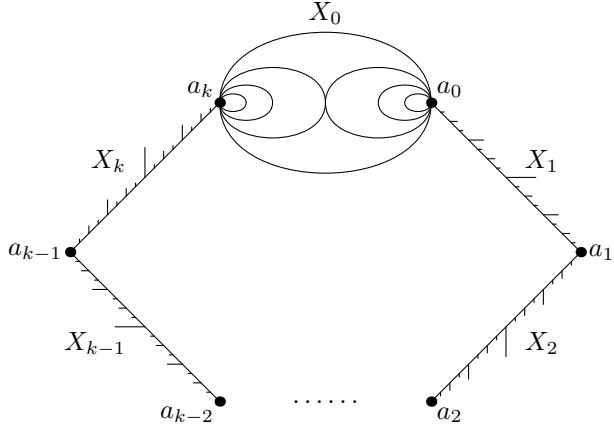}
 \caption{A continuum from Lemma~\ref{T:I(free-arc)}}
 \label{Fig:freeArc}
\end{figure}

\begin{proof}
We may assume that
$\Cut_{X_0}(a_k,a_0)$ is uncountable, for if not,
we replace $k$ by $k'=k-1$ and $X_{k-1}$ by $X'_{k-1}=X_{k-1}\cup X_k$
(if $k'=1$ we simply apply Lemma~\ref{P:exactLipschitz}).

Consider the metrics $d_i=d_{X_i,a_{i-1},a_i}$ for $i=0,\dots,k$
(with $a_{-1}:=a_k$); in view of (c) we may assume that
$f_i:(X_i,d_i)\to (X_{i+1},d_{i+1})$ are isometries for $1\le i<k$.
Let $d$ be the convex metric on $X$ defined by (\ref{EQ:convexMetricOnUnion}).
Since $\bigcup_{i\ne j} X_i\cap X_j = \{a_i:\ i=0,\dots,k\}$,
$1>d_i(a_{i-1},a_i)>1/2$
and $X_i\cap X_j$ are empty or singletons for $i\ne j$,
Lemma~\ref{L:convexMetricOnUnion}
gives that $d|_{X_i\times X_i}=d_i$ for every $i$.

Let $f_k:X_k\to X_0$, $f_0:X_0\to X_1$ and $g_0:X_0\to (X_1\cup X_2)$
be maps from Proposition~\ref{T:lipschitzXtoX} with ($\varrho>2$)
such that $f_0(a_k)=g_0(a_k)=a_0$, $f_0(a_0)=g_0(a_0)=a_1$,
$f_k(a_{k-1})=a_k$ and $f_k(a_k)=a_0$. Define
$f,g:X\to X$ by
$$
  f(x)=\begin{cases}
   f_0(x)  &\text{if } x\in X_0;
   \\
   h_i(x)  &\text{if } x\in X_i,\ 1\le i<k;
   \\
   f_k(x)  &\text{if } x\in X_k;
  \end{cases}
\qquad
  g(x)=\begin{cases}
   g_0(x)  &\text{if } x\in X_0;
   \\
   h_i(x)  &\text{if } x\in X_i,\ 1\le i<k;
   \\
   f_k(x)  &\text{if } x\in X_k.
  \end{cases}
$$
Analogously as in the proof of Lemma~\ref{T:I(omega-star)}
we can show that $f$ is (not totally) Devaney chaotic,
$g$ is exact and the entropies
of $f,g$ are positive and bounded from above by $c/k$, where
$c$ is an absolute constant. What is left is to prove that $g$ has dense periodic points.

To this end put $P=\{a_i:\ i=0,\dots,k\}$ and realize that $g(P)=P$.
Let $\DDd$ be the decomposition of $X$ into $P$ and singletons $\{x\}$, $x\not\in P$.
Let $X'=X/\DDd$ (i.e.~we collapse $P$ into a point),
$\pi:X\to X'$ be the natural projection
and $\{a'\}=\pi(P)$, $X_i'=\pi(X_i)$.
Then ${X'}$ is a compact metric space since
the decomposition $\DDd$ is upper semicontinuous.
The map $g:X\to X$ induces ${g'}:{X'}\to {X'}$
which fixes ${a'}$ and, being a factor of $g$,
is exact. Notice that $g'|_{X_0'}:X_0'\to X_1'\cup X_2'$
can be written as the composition $\varphi\circ \psi$
with continuous $\psi:X_0'\to\III$ and $\varphi:\III\to X_1'\cup X_2'$.
So, analogously as in the proof of Lemma~\ref{T:I(omega-star)},
$g'$ has dense periodic points.
Hence also $g$ has dense periodic points, since whenever
$x'\ne a'$ is a periodic point of $g'$, then its only $\pi$-preimage $x$
is a periodic point of $g$.
\end{proof}

Finally, we are ready to prove our main result.

\begin{theoremMain}
\theoremMainText{}
\end{theoremMain}

\begin{proof}
Note that if $X$ contains a non-disconnecting free arc $A$
then, for every $k\ge 2$, $X$ can be written as in Lemma~\ref{T:I(free-arc)}
with $X_1,\dots,X_k$ being subarcs of $A$.
So the theorem immediately follows from Corollary~\ref{C:IT(X)_is_finite}
and Lemmas~\ref{T:I(omega-star)}, \ref{T:I(free-arc)}.
\end{proof}

\begin{remark}\label{R:generalizations}
Lemma~\ref{T:I(free-arc)}
gives $\IED(X)=0$ also for some
continua containing no non-disconnecting free arc.
For example if, for some $m\in\{3,4,\dots,\aleph_0\}$,
$X$ contains the universal dendrite $D_m$ of order $m$ with
two point boundary and with connected $X\setminus \interior{D_m}$,
we again have $\IED(X)=0$.
(Recall that $D_m$ is the topologically unique dendrite
such that the branch points of it are dense and all have the order $m$.)
\end{remark}


\end{document}